\documentclass[12pt,reqno]{amsart}
\usepackage{fullpage}
\usepackage[top=1.5in, bottom=2in, left=2.5in, right=2.5in]{geometry}
\usepackage{graphicx,amsmath,amssymb,amsfonts,amsthm,enumitem,bm,xcolor}
\usepackage{fullpage}
\usepackage{cleveref}

\usepackage[normalem]{ulem}

\usepackage{url}

\usepackage{color}

\newcommand{\N}{\mathbb{N}}

\renewcommand{\a}{\alpha}
\renewcommand{\b}{\beta}

\renewcommand{\o}{{\omega}}

\renewcommand{\(}{\left\(}
\renewcommand{\)}{\right\)}

\newcommand{\pa}[2]{\left(\frac{#1}{#2}\right)}

\numberwithin{equation}{section}
\theoremstyle{plain}
\newtheorem{theorem}{Theorem}[section]
\newtheorem{lemma}[theorem]{Lemma}

\newtheorem{conjecture}[theorem]{Conjecture}
\newtheorem*{conjecture*}{Conjecture}
\newtheorem{corollary}[theorem]{Corollary}

\theoremstyle{definition}

\numberwithin{equation}{section}

\setlist[enumerate]{leftmargin=*,label=\rm{(\arabic*)}}

\makeatletter
\@namedef{subjclassname@2020}{%
	\textup{2020} Mathematics Subject Classification}
\makeatother

\allowdisplaybreaks

\title{On a conjecture of Andrews and Bachraoui}
\author{Koustav Banerjee}
\author{Kathrin Bringmann}
\author{William Keith}
\address{University of Cologne, Department of Mathematics and Computer Science, Weyertal 86-90, 50931 Cologne, Germany}
\email{kbanerj1@uni-koeln.de}
\email{kbringma@uni-koeln.de}

\address{Department of Mathematical Sciences, Michigan Tech, Houghton, MI 49931-1295}
\email{wjkeith@mtu.edu}

\makeatletter
\@namedef{subjclassname@2020}{%
	\textup{2020} Mathematics Subject Classification}
\makeatother


\begin{document}
	
	\begin{abstract}
Recently, Andrews and Bachraoui considered a generating function $F_{k,m}(q)$ associated with certain two-color partitions, and conjectured that this function has non-negative coefficients for $m=1$. They showed this property for $1 \leq k \leq 4$.  In this note, we prove that $F_{k,1}(q)$ has non-negative coefficients for $5 \leq k \leq 10$. Moreover, we show that, as $k\to\infty$, $F_{k,1}(q)$ is related to Ramanujan's third order mock theta function $\o(q)$ and to quotients of certain $q$-binomial coefficients.
	\end{abstract}
	
	\maketitle



\section{Introduction and Statement of Results}
 Throughout let $(a)_n=(a;q)_n:=\prod_{j=0}^{n-1} (1-aq^j)$ for $n\in\N_0\cup\{\infty\}$. Moreover let $(a_1,...,a_\ell)_n=(a_1,\dots,a_\ell;q)_n:=(a_1)_n\cdots(a_\ell)_n$ for $\ell\in\N$. 
Andrews and Bachraoui \cite[Conjecture 2.8 and Conjecture 2.9]{AB} conjectured that for $k,m\in \mathbb{N}$ and for $k=2, m=4$, the coefficients of the $q$-series 
\begin{equation*}
\sum_{n\ge 0} \frac{\left(q^{2n+2},q^{2n+2k};q^2\right)_\infty}{\left(q^{2n+1};q^2\right)^2_\infty} q^{m(2n+1)}=:\sum_{n\ge 0}c_{k,m}(n)q^n=:F_{k,m}(q)
\end{equation*}
are positive. Note that for $m=1$, 
we have 
\begin{equation}\label{def}
F_{k,1}(q)=\sum_{n\ge 0}\frac{\left(q^{2n+2},q^{2n+2k};q^2\right)_\infty}{\left(q^{2n+1};q^2\right)^2_\infty} q^{2n+1}.
\end{equation}

Andrews and Bachraoui \cite[(2.2), (2.3), (2.5), Theorem 2.6]{AB} proved that $c_{k,1}(n)>0$ for $n\in \mathbb{N}_0$ and $1\le k\le 4$. Moreover they proposed the following conjecture (see \cite[Conjecture 2.8]{AB}), which we partially prove in this paper. 

\begin{conjecture}\label{ABconj}
	For $k\in \mathbb{N}$ and $n\in \mathbb{N}_0$, $c_{k,1}(n)>0$.	
\end{conjecture}

 To state our theorems, for $A(q)=\sum_{n\ge 0}a(n)q^n$ we mean, by $A(q)\succeq 0$ (resp. $A(q)\succ 0$) that $a(n)\ge 0$ (resp. $a(n)> 0$) for $n\in \mathbb{N}_0$.


\begin{theorem}\label{thm1}
	For $k\in \{5,6,7\}$,  we have $F_{k,1}(q)\succeq 0$.
\end{theorem}

Using a computational approach we extend this to $k \leq 10$.

\begin{theorem}\label{thm2} For $k \in \{8,9,10\}$, we have $F_{k,1}(q) \succeq 0$. \end{theorem}

Our next result relates $F_{k,1}$ to the mock theta function $\omega$ of Ramanujan \cite{AB5}
\begin{equation}\label{Ramanujanomega}
\omega(q):=\sum_{n\ge0}\frac{q^{2n(n+1)}}{\left(q;q^2\right)^2_{n+1}}.
\end{equation}
 \begin{theorem}\label{lem2}
	We have
	\[
	\lim_{k\to\infty} F_{k,1}(q)= q\omega(q).
	\]
\end{theorem}

In the following theorem, we obtain more information about the function that prevents $F_{k,1}(q)$ from agreeing with $q\o(q)$.  To state our result, for $A(q)=\sum_{n\ge 0}a(n)q^n$ and $m\in \mathbb{N}_0$, define
\[
\operatorname{coeff}_{\left[q^m\right]}A(q):=a(m).
\]
\begin{theorem}\label{lem3}
	For $k\in\mathbb{N}$, we have
	\[
	q\omega(q)-F_{k,1}(q)=q^{2k+1}E_k(q)\ \text{with}\ E_k(q)\in \mathbb{Z}[[q]]\ \text{and}\ \operatorname{coeff}_{\left[q^0\right]}E_k(q)=1.
	\]	
\end{theorem}
Theorem \ref{lem3} immediately implies the following.
\begin{corollary}\label{cor1}
	For $k\in\mathbb{N}$ and $1\le m\le k$, we have
	\[
	\operatorname{coeff}_{\left[q^m\right]}\left(q\omega(q)-F_{k,1}(q)\right)=0\ \ \text{for}\ \ 1\le m\le 2k.
	\]
\end{corollary}

The remainder of the paper is organized as follows. In Section \ref{sec:prelim} we recall some standard $q$-series transformations and define notation. In Section \ref{sec:formulas} we transform $F_{k,1}(q)$ into certain forms which assist in proving Theorems \ref{thm1}, \ref{lem2}, and \ref{lem3}. In Section \ref{mainresults}, we prove Theorem \ref{thm1} and Theorem \ref{thm2}, showing the positivity of the coefficients of $F_{k,1}(q)$ for $5\le k \leq 10$. Next, we prove \Cref{lem2}, and \Cref{lem3}, exhibiting the relations between $F_{k,1}$ and the third order mock theta function $\omega$. Finally, in \Cref{conc}, we presents a few plausible approaches to address \Cref{ABconj}.

\section*{Acknowledgements}
The first and the second author have received funding from the European Research Council (ERC)
under the European Union's Horizon 2020 research and innovation programme (grant agreement
No. 101001179). The authors thank Caner Nazaroglu for his helpful comments and suggestions.

\section{Preliminaries}\label{sec:prelim}
Recall Heine's transformation \cite[Corollary 1.2.4]{AndB}. 
\begin{lemma}\label{HeineTransform}
	For $|q|, |t|<1$, $0<|c|<|b|<1$, we have
	\[
	\sum_{n\ge 0}\frac{(a,b)_nt^n}{(q,c)_n}=\frac{\left(\frac cb,bt\right)_{\infty}}{(c,t)_{\infty}}\sum_{n\ge 0}\frac{\left(\frac{abt}{c},b\right)_n\pa{c}{b}^n}{(q,bt)_n}.
	\]
\end{lemma}
By \cite[Theorem 1]{A0}, we have the following transformation.
\begin{lemma}\label{ThmA0}
	We have
	\begin{align*}
	\sum_{n\ge 0}\frac{\left(B,-Abq\right)_nq^n}{\left(-aq,-bq\right)_n}&=-\frac{a^{-1}\left(B,-Abq\right)_{\infty}}{\left(-aq,-bq\right)_{\infty}}\sum_{n\ge 0}\frac{\left(A^{-1}\right)_n\pa{Abq}{a}^n}{\left(-\frac{B}{a}\right)_{n+1}}\\
	&\hspace{2.5 cm}+\left(1+b\right)\sum_{n\ge 0}\frac{\left(-a^{-1}\right)_{n+1}\left(-\frac{ABq}{a}\right)_n\left(-b\right)^n}{\left(-\frac{B}{a},\frac{Abq}{a}\right)_{n+1}}.
	\end{align*}
\end{lemma}

We next state the Rogers--Fine identity \cite[equation (2.7)]{A0}.
\begin{lemma}\label{RF}
	We have
	\begin{equation*}
	\sum_{n\ge 0}\frac{\left(\a\right)_nw^n}{\left(\beta\right)_n}=\sum_{n\ge 0}\frac{\left(\a,\frac{\a wq}{\b}\right)_n\b^nw^nq^{n^2-n}\left(1-\a w q^{2n}\right)}{\left(\b\right)_n\left(w\right)_{n+1}}.
	\end{equation*}
\end{lemma}

The \emph{Gaussian polynomial} is defined by (see \cite[Definition 3.1]{Andbook})
\begin{equation*}
\begin{bmatrix}
n\\m
\end{bmatrix}_q:=\begin{cases}
\frac{(q)_n}{(q)_m(q)_{n-m}} &\quad \text{if}\ 0\le m\le n,\\
0 &\quad \text{otherwise}.
\end{cases}
\end{equation*}
From \cite[equation (3.3.6)]{Andbook}, we have, for $n\in \mathbb{N}$,
\begin{equation}\label{Gaussian}
\left(a;q\right)_N=\sum_{j=0}^{N}\begin{bmatrix}
N\\
j
\end{bmatrix}_{q}(-1)^j a^j q^{\frac{j^2-j}{2}}.
\end{equation}

\section{Some $q$-series transformations}\label{sec:formulas}
The following lemma is a key step for proving Theorem \ref{thm1}.

\begin{lemma}\label{newlem1}
	For $k\in \mathbb{N}_{\ge 3}$, we have
	\[
	F_{k,1}(q)=\frac{q}{\left(1-q\right)\left(q;q^2\right)_{k-1}}\left(1+\left(1-q\right)\sum_{n=1 }^{k-2}\frac{\left(q^{4-2k};q^2\right)_n}{\left(q^2;q^2\right)_n}\frac{q^{(2k-1)n}}{1-q^{2n+1}}\right).
	\]	
\end{lemma}
\begin{proof}
	Using \eqref{def}, we obtain 
	\begin{align}\label{eqn1}
	F_{k,1}(q)
	&=\frac{q\left(q^2;q^2\right)^2_{\infty}}{\left(q;q^2\right)^2_{\infty}\left(q^2;q^2\right)_{k-1}}\sum_{n\ge 0}\frac{\left(q;q^2\right)^2_nq^{2n}}{\left(q^2;q^2\right)_n\left(q^{2k};q^2\right)_{n}}.
	\end{align}
	Applying Lemma \ref{HeineTransform} with $q\mapsto q^2$, $a=b=q$, $c=q^{2k}$, and $t=q^2$, we have
	\begin{equation*}
	\sum_{n\ge 0}\frac{\left(q;q^2\right)^2_nq^{2n}}{\left(q^2,q^{2k};q^2\right)_{n}}=\frac{\left(q^{2k-1},q^3;q^2\right)_{\infty}}{\left(q^{2k},q^2;q^2\right)_{\infty}}\sum_{n\ge 0}\frac{\left(q^{4-2k},q;q^2\right)_n}{\left(q^2,q^3;q^2\right)_n}q^{(2k-1)n}.
	\end{equation*}
	Plugging this into \eqref{eqn1}, it follows that 
	\begin{align}\label{eqn2}
	F_{k,1}(q)
	&\!=\frac{q}{\left(1-q\right)\left(q;q^2\right)_{k-1}}\left(1+\left(1-q\right)\sum_{n\ge 1}\frac{\left(q^{4-2k};q^2\right)_n}{\left(q^2;q^2\right)_n}\frac{q^{(2k-1)n}}{1-q^{2n+1}}\right).
	\end{align}
	Note that for $n\ge k-1$, $(q^{4-2k};q^2)_n=
	0$. Therefore, by \eqref{eqn2}, we have for $k\ge 3$
	\begin{align*}
	F_{k,1}(q)=\frac{q}{\left(1-q\right)\left(q;q^2\right)_{k-1}}\left(1+\left(1-q\right)\sum_{n=1 }^{k-2}\frac{\left(q^{4-2k};q^2\right)_n}{\left(q^2;q^2\right)_n}\frac{q^{(2k-1)n}}{1-q^{2n+1}}\right).
	\end{align*}
	This concludes the proof.
\end{proof}

Next, we present two lemmas which are used in the proof of Theorem \ref{lem3}. 

\begin{lemma}\label{lem1}
	For $k\in \mathbb{N}$, we have
	\[
	F_{k,1}(q)=\sum_{n\ge 0}\frac{\left(q^{2k-1};q^2\right)_nq^{n+1}}{\left(q;q^2\right)_{n+1}}.
	\]	
\end{lemma}
\begin{proof}
	By \eqref{eqn1}, we have
	\begin{equation}\label{eqn3}
	F_{k,1}(q)=\frac{\left(q^2;q^2\right)^2_{\infty}}{(q;q^2)^2_{\infty}\left(q^2;q^2\right)_{k-1}}\sum_{n\ge 0}\frac{\left(q;q^2\right)^2_nq^{2n+1}}{\left(q^2,q^{2k};q^2\right)_n}.
	\end{equation}
	Taking $q\mapsto q^2$, $a=-1, b=-q^{2k-2}, A=q^{1-2k}$, and $B=q$ in Lemma \ref{ThmA0}, using the fact that for $n\in\N_0$, $(1;q^2)_{n+1}=0$, we obtain
	\[
	\sum_{n\ge 0}\frac{\left(q;q^2\right)^2_nq^{2n+1}}{\left(q^2,q^{2k};q^2\right)_n}=\frac{(q;q^2)^2_{\infty}}{\left(q^2,q^{2k};q^2\right)_{\infty}}\sum_{n\ge 0}\frac{\left(q^{2k-1};q^2\right)_n q^{n+1}}{\left(q;q^2\right)_{n+1}}.
	\] 
	Plugging this into \eqref{eqn3}, we get
	\begin{align*}
	F_{k,1}(q)&=\frac{\left(q^2;q^2\right)^2_{\infty}}{\left(q;q^2\right)^2_{\infty}\left(q^2;q^2\right)_{k-1}}\frac{(q;q^2)^2_{\infty}}{\left(q^2,q^{2k};q^2\right)_{\infty}}\sum_{n\ge 0}\frac{\left(q^{2k-1};q^2\right)_n q^{n+1}}{\left(q;q^2\right)_{n+1}}\\&=\sum_{n\ge 0}\frac{\left(q^{2k-1};q^2\right)_n q^{n+1}}{\left(q;q^2\right)_{n+1}}.
	\end{align*}
	This gives the lemma.
\end{proof}



\section{Proof of Theorems \ref{thm1}, \ref{thm2}, \ref{lem2}, and \ref{lem3}}\label{mainresults}
We are now ready to prove Theorem \ref{thm1}.

\noindent \emph{Proof of Theorem \ref{thm1}}. We show Theorem \ref{thm1} in three parts. First we prove the claim for $k=5$. By Lemma \ref{newlem1}, we have
\begin{align*}
F_{5,1}(q)&=\frac{q}{\left(1-q\right)\left(q;q^2\right)_{4}}\left(1+\left(1-q\right)\sum_{n=1 }^{3}\frac{\left(q^{-6};q^2\right)_n}{\left(q^2;q^2\right)_n}\frac{q^{9n}}{1-q^{2n+1}}\right)  \\    
&= \frac{q}{(1-q)\left(1-q^3\right)\left(1-q^5\right)\left(1-q^7\right)}  \left(\frac{1}{1-q}+\frac{1-q^{-6}}{1-q^2}\frac{q^9}{1-q^3} \right.\\
&\hspace{.5cm}\left. +\frac{\left(1-q^{-6}\right)\left(1-q^{-4}\right)}{\left(1-q^2\right)\left(1-q^4\right)}\frac{q^{18}}{1-q^5}+\frac{\left(1-q^{-6}\right)\left(1-q^{-4}\right)\left(1-q^{-2}\right)}{\left(1-q^2\right)\left(1-q^4\right)\left(1-q^6\right)}\frac{q^{27}}{1-q^7}\right).
\end{align*}
Simplifying the factors in the summands on the right-hand side, we obtain
\begin{multline*}
F_{5,1}(q)=\frac{q}{(1-q)\left(1-q^3\right)\left(1-q^5\right)\left(1-q^7\right)} \\ \times\left(\frac{1}{1-q}-\left(1+q^2+q^4\right)\frac{q^3}{1-q^3}+\left(1+q^2+q^4\right)\frac{q^8}{1-q^5}-\frac{q^{15}}{1-q^7}\right).
\end{multline*}

The claim follows once we prove that
\begin{equation*}
\frac{1}{1-q}\left(\frac{1}{1-q}-\frac{\left(1+q^2+q^4\right)q^3}{1-q^3}+\frac{\left(1+q^2+q^4\right)q^8}{1-q^5}-\frac{q^{15}}{1-q^7}\right) \succeq 0.
\end{equation*}
We now show that we can treat the first and the second together as well as the third and fourth. First we write
\begin{equation*}
\frac{\left(1+q^2+q^4\right)q^3}{1-q^3} = -1-q-q^2-q^4+\frac{1+q+q^2}{1-q^3}.
\end{equation*}
Then
\begin{equation*}
\frac{1}{1-q} + 1 + q + q^2 + q^4 - \frac{1+q+q^2}{1-q^3} \succeq 0.
\end{equation*}
Next we write
\begin{equation*}
\frac{\left(1+q^2+q^4\right)q^8}{1-q^5} - \frac{q^{15}}{1-q^7} = q^8 \bigg(\frac{1+q^2+q^4}{1-q^5}+1-\frac{1}{1-q^7}\bigg).
\end{equation*}
Now
\begin{equation*}
\frac{1+q^2+q^4}{1-q^5} - \frac{1}{1-q^7} \succeq \frac{1}{1-q^5} - \frac{1}{1-q^7} = \frac{q^5-q^7}{\left(1-q^5\right)\left(1-q^7\right)} = \frac{q^5\left(1-q^2\right)}{\left(1-q^5\right)\left(1-q^7\right)}.
\end{equation*}
Dividing by $\frac{1}{1-q}$ gives
\begin{equation*}
\frac{q^5(1+q)}{(1-q^5)(1-q^7)} \succeq 0.
\end{equation*}
This finishes the proof of the case $k=5$. 

Next, we show the case $k=6$. By Lemma \ref{newlem1}, we have
\begin{equation*}
F_{6,1}(q)  = \frac{q}{(1-q)(q;q^2)_5}\left(1+(1-q)\sum_{n=1}^4 \frac{\left(q^{-8};q^2\right)_n}{(q^2;q^2)_n} \frac{q^{11n}}{1-q^{2n+1}}\right).
\end{equation*}
Simplifying the factors $\frac{(q^{-8};q^2)_n}{(q^2;q^2)_n}$ (for $1\le n\le 4$) in the sum on the right-hand side of the above identity, we have
\begin{multline*}
F_{6,1}(q) = \frac{q}{(q;q^2)_5} \left(\frac{1}{1-q} - \frac{q^3\left(1+q^2+q^4+q^6\right)}{1-q^3} + \frac{q^8\left(1+q^2+q^4\right)\left(1+q^4\right)}{1-q^5}\right.\\
\left.- \frac{q^{15}\left(1+q^2+q^4+q^6\right)}{1-q^7} + \frac{q^{24}}{1-q^9}\right).
\end{multline*}
Using long division, we write
\begin{align*}
&-\frac{q^3 \left(1+q^2+q^4+q^6\right)}{1-q^3} =  2+q+q^2+q^3+q^4+q^6-\frac{2+q+q^2}{1-q^3},\\
&\frac{q^8\left(1+q^2+q^4\right)\left(1+q^4\right)}{1-q^5}=   -1-q-2q^2-q^3-q^4-q^5-q^6-2q^7-q^9-q^{11}\\
&\hspace{9.5 cm}+\frac{1+q+2q^2+q^3+q^4}{1-q^5},\\
& - \frac{q^{15} \left(1+q^2+q^4+q^6\right)}{1-q^7}   =   1+q+q^3+q^5+q^7+q^8+q^{10}+q^{12}+q^{14} \\
&\hspace{10.6 cm}- \frac{1+q+q^3+q^5}{1-q^7}.
\end{align*}
Thus, we need
\begin{align*}
&\frac{1}{1-q} \left(\frac{1}{1-q}+2+q+q^2+q^3+q^4+q^6- \frac{2+q+q^2}{1-q^3}  - 1-q-2q^2-q^3-q^4  \right.\\
&\left.  -q^5-q^6-2q^7-q^9-q^{11} + \frac{1+q+2q^2+q^3+q^4}{1-q^5} +1+q+q^3+q^5+q^7+q^8 \right.\\
&\hspace{7.4 cm}\left. +q^{10}+q^{12}+q^{14}- \frac{1+q+q^3+q^5}{1-q^7}\right)\\
&= \frac{1}{1-q} \left(\frac{1}{1-q} + 2+q-q^2+q^3-q^7+q^8-q^9+q^{10}-q^{11}+q^{12}+q^{14} \right.\\
&\hspace{2.85cm}\left. -  \frac{2+q+q^2}{1-q^3}  +\frac{1+q+2q^2+q^3+q^4}{1-q^5}  - \frac{1+q+q^3+q^5}{1-q^7} \right) \succeq 0.
\end{align*}
Now we write 
\begin{align*}
\frac{1+q+2q^2+q^3+q^4}{1-q^5} &= \frac{1}{1-q} + \frac{q^2}{1-q^5}, \quad
-\frac{2+q+q^2}{1-q^3} = -\frac{1}{1-q^3} - \frac{1}{1-q},\\
- \frac{1+q+q^3+q^5}{1-q^7} &= - \frac{1}{1-q} + \frac{q^2+q^4+q^6}{1+q^7}.
\end{align*}
So we want
\begin{align}\label{PO}
&\hspace{-.2cm}\frac{1}{1-q}\Biggl(2+q-q^2+q^3-q^7+q^8-q^9+q^{10}-q^{11}+q^{12}+q^{14}- \frac{1}{1-q^3} \nonumber\\
&\hspace{6.4cm}+ \frac{q^2}{1-q^5}+\frac{q^2\left(1+q^2+q^4\right)}{1-q^7}\Biggr) \succeq 0.
\end{align}
To show this, first note that
\begin{align}\label{PO1}
&\frac{1}{1-q}\!\left(- \frac{1}{1-q^3}\!+\! \frac{q^2}{1-q^5}\!+\!\frac{q^2\left(1+q^2+q^4\right)}{1-q^7}\right)\nonumber\\
&\hspace{5.2 cm}\!\succeq\! \frac{1}{1-q}\!\left(\!- \frac{1}{1-q^3}\!+\!\frac{q^2}{1-q^5}\!+\!\frac{q^2\left(1+q^2\right)}{1-q^7}\right).
\end{align}
Now we prove that
\begin{equation}\label{P0}
\frac{1}{1-q}\!\left(- \frac{1}{1-q^3} + \frac{q^2}{1-q^5}+\frac{q^2\left(1+q^2\right)}{1-q^7}\right)\succeq -\frac{1+q}{1-q}.
\end{equation} 
This is equivalent to showing that
\begin{equation}\label{PO2}
\frac{1}{1-q}\!\left(1+q- \frac{1}{1-q^3} + \frac{q^2}{1-q^5}+\frac{q^2\left(1+q^2\right)}{1-q^7}\right)\succeq 0.
\end{equation}
We have
\begin{align*}
&\hspace{-1 cm}\frac{1}{1-q}\!\left(1+q- \frac{1}{1-q^3}+\frac{q^2}{1-q^5}+\frac{q^2\left(1+q^2\right)}{1-q^7}\right)\\
&=\frac{1}{1-q}\!\left(1+q- \frac{1-q^2}{\left(1-q^3\right)\left(1-q^5\right)}+\frac{q^2\left(1+q^2\right)}{1-q^7}\right)\\
&=\frac{1+q}{1-q}-\frac{1+q}{\left(1-q^3\right)\left(1-q^5\right)}+\frac{q^2+q^4}{\left(1-q\right)\left(1-q^7\right)}\\
&=\frac{1+q}{1-q}-\frac{1+q}{\left(1-q^3\right)\left(1-q^5\right)}+\frac{q^2-q^4+2q^4}{\left(1-q\right)\left(1-q^7\right)}\\
&=\frac{1+q}{1-q}-\frac{1+q}{\left(1-q^3\right)\left(1-q^5\right)}+\frac{q^2\left(1+q\right)}{1-q^7}+\frac{2q^4}{\left(1-q\right)\left(1-q^7\right)}.
\end{align*}
Since $\frac{q^2\left(1+q\right)}{1-q^7}\succeq 0$, therefore, to prove \eqref{PO2}, it is enough to show that
\begin{equation}\label{PO3}
\frac{1+q}{1-q}-\frac{1+q}{\left(1-q^3\right)\left(1-q^5\right)}+\frac{2q^4}{\left(1-q\right)\left(1-q^7\right)}\succeq 0.
\end{equation}
Simplifying the left-hand side of \eqref{PO3}, we get
\begin{align*}
&\frac{1+q}{1-q}\!-\!\frac{1+q}{\left(1-q^3\right)\left(1-q^5\right)}\!+\!\frac{2q^4}{\left(1-q\right)\left(1-q^7\right)}\\
&\hspace{2.5 cm}=\!\frac{\left(1+q\right)\left(q+q^2-q^5-q^6-q^7\right)}{\left(1-q^3\right)\left(1-q^5\right)}\!+\!\frac{2q^4}{\left(1-q\right)\left(1-q^7\right)}\\
&\hspace{2.5 cm}= \frac{\left(1+q\right)\left(q\left(1-q^6\right)+q^2\left(1-q^3\right)-q^6\right)}{\left(1-q^3\right)\left(1-q^5\right)}+\frac{2q^4}{\left(1-q\right)\left(1-q^7\right)}\\
&\hspace{2.5 cm}\succeq q^4\left(-\frac{q^2+q^3}{\left(1-q^3\right)\left(1-q^5\right)}+\frac{2}{\left(1-q\right)\left(1-q^7\right)}\right)\\
&\hspace{2.5 cm}=q^4\Bigg(\frac{2q}{\left(1-q^3\right)\left(1-q^7\right)}+\frac{2\left(1+q^2\right)}{\left(1-q^3\right)\left(1-q^7\right)}-\frac{q^2+q^3}{\left(1-q^3\right)\left(1-q^5\right)}\Bigg).
\end{align*}
Since $\frac{2q}{\left(1-q^3\right)\left(1-q^7\right)}\succeq 0$, thus, to prove \eqref{PO3}, it suffices to show that 
\begin{equation*}
\frac{2\left(1+q^2\right)}{\left(1-q^3\right)\left(1-q^7\right)}-\frac{q^2+q^3}{\left(1-q^3\right)\left(1-q^5\right)}\succeq 0.
\end{equation*}
To prove this, we show that
\begin{align}\label{PO4}
\frac{1+q^2}{\left(1-q^3\right)\left(1-q^7\right)}-\frac{q^2}{\left(1-q^3\right)\left(1-q^5\right)}\succeq 0,\\\label{PO5}
\frac{1+q^2}{\left(1-q^3\right)\left(1-q^7\right)}-\frac{q^3}{\left(1-q^3\right)\left(1-q^5\right)}\succeq 0.
\end{align}
First, we see that
\begin{align*}
\frac{1+q^2}{\left(1-q^3\right)\left(1-q^7\right)}-\frac{q^2}{\left(1-q^3\right)\left(1-q^5\right)}
=\frac{1}{1-q^3}+\frac{q^9}{\left(1-q^5\right)\left(1-q^7\right)}\succeq 0,
\end{align*}
which proves \eqref{PO4}. 

Next, we have
\begin{align*}
\frac{1+q^2}{\left(1-q^3\right)\left(1-q^7\right)}-\frac{q^3}{\left(1-q^3\right)\left(1-q^5\right)}
=\frac{1}{\left(1-q^5\right)\left(1-q^7\right)}+\frac{q^2}{1-q^7}\succeq 0.
\end{align*}
This gives \eqref{PO5}. This concludes the proof of \eqref{P0}. 

Thus, combining \eqref{PO1} and \eqref{P0}, it follows that
\begin{equation*}
\frac{1}{1-q}\!\left(- \frac{1}{1-q^3}\!+\! \frac{q^2}{1-q^5}\!+\!\frac{q^2\left(1+q^2+q^4\right)}{1-q^7}\right)\!\succeq -\frac{1+q}{1-q}.
\end{equation*}
Thus, to prove \eqref{PO}, it remains to show that
\[
\frac{1}{1-q}\left(2+q-q^2+q^3-q^7+q^8-q^9+q^{10}-q^{11}+q^{12}+q^{14}-\left(1+q\right)\right)\succeq 0,
\]
which holds. This gives the claim for $k=6$.

Finally, we consider the case $k=7$. From Lemma \ref{newlem1}, we have
\begin{equation*}
F_{7,1}(q) = \frac{q}{\left(1-q\right)\left(q;q^2\right)_6} \left(1+(1-q) \sum_{n=1}^{5} \frac{\left(q^{-10};q^2\right)_n}{\left(q^2;q^2\right)_n} \frac{q^{13n}}{1-q^{2n+1}}\right).
\end{equation*}
Simplifying the factors $\frac{(q^{-10};q^2)_n}{(q^2;q^2)_n}$ (for $1\le n\le 5$) in the sum on the right hand side of the above identity, we have 
\begin{align*}
F_{7,1} (q) &= \frac{q}{\left(q;q^2\right)_6} \left(\frac{1}{1-q}-\frac{q^3\left(1+q^2+q^4+q^6+q^8\right)}{1-q^3}\right.\\ &\hspace{-1 cm}\quad \, \left. + \frac{q^8 \left(1+q^4\right)\left(1+q^2+q^4+q^6+q^8\right)}{1-q^5}-\frac{q^{15}\left(1+q^4\right)\left(1+ q^2+q^4+q^6+q^8\right)}{1-q^7} \right.\\
&\hspace{4.7 cm}\quad \, \left. + \frac{q^{20} \left(1+q^2+q^4+q^6+q^8\right)}{1-q^9} - \frac{q^{35}}{1-q^{11}}\right). 
\end{align*}

Similar to the case $k=6$, after carrying long divisions, to prove the claim, we need to show that
\begin{align*}
&\frac{1}{\left(q;q^2\right)_6} \Biggl(2 - q + 2q^2 - q^4 + q^5 - q^6 - q^{10} - 2 q^{11} + 2 q^{12} + 2 q^{14} - 2 q^{15}\\
&\hspace{3 cm}  + 2 q^{16} - q^{17}+ q^{18} - q^{19} + q^{20}+q^{24}+\frac{1}{1-q}-\frac{2+q+2 q^2}{1-q^3}\\
&\hspace{3 cm} +\frac{2 + 2 q + 2 q^2 + 2 q^3 + 2 q^4}{1-q^5}-\frac{2 +q +2 q^2 +q^3 +q^4 +2 q^5 + q^6}{1-q^7}\\
&\hspace{6.5 cm} +\frac{q + q^2 + q^4 + q^6 + q^8}{1-q^9}-\frac{q^{2}}{1-q^{11}} \Biggr)\succeq 0.
\end{align*}
Now we write
\begin{align*}
-\frac{2+q+2q^2}{1-q^3} &=- \frac{1}{1-q} -\frac{1+q^2}{1-q^3},\\
\frac{2+2q+2q^2+2q^3+2q^4}{1-q^5} &= \frac{2}{1-q},\\
- \frac{2+q+2q^2+q^3+q^4+2q^5+q^6}{1-q^7} &= - \frac{1}{1-q} - \frac{1+q^2+q^5}{1-q^7},\\
\frac{q + q^2 + q^4 + q^6 + q^8}{1-q^9}&=\frac{1}{1-q}-\frac{1+q^3+q^5+q^7}{1-q^9}.
\end{align*}
Thus, note that the claim follows if we show 
\begin{align*}
&\frac{1}{1-q} \Biggl(2 - q + 2q^2 - q^4 + q^5 - q^6 - q^{10} - 2 q^{11} + 2 q^{12} + 2 q^{14} - 2 q^{15} + 2 q^{16}\nonumber\\
&\hspace{0.5 cm}-\!q^{17}\!+\!q^{18}\!-\!q^{19}\!-\!\frac{1\!+\! q^2}{1\!-\!q^3}\!+\!\frac{2}{1\!-\!q}\!-\!\frac{1\!+\! q^2\!+\!q^5}{1\!-\!q^7}\!-\!\frac{1\!+\!q^3\!+\!q^5\!+\!q^7}{1\!-\!q^9}\!-\!\frac{q^{2}}{1\!-\!q^{11}} \Biggr)\succeq 0.
\end{align*}
Next, we simplify
\begin{align*}
-\frac{1+ q^2}{1-q^3}+\frac{2}{1-q}-\frac{1+q^3 + q^5 + q^7}{1-q^9}
&=\frac{2 q + q^2 + 2 q^4 + q^6 + q^7 + q^8}{1-q^9}.
\end{align*}
Thus, we need to show that
\begin{align}\label{eqn7b}
&\frac{1}{1-q} \Biggl(2 - q + 2q^2 - q^4 + q^5 - q^6 - q^{10} - 2 q^{11} + 2 q^{12} + 2 q^{14} - 2 q^{15} + 2 q^{16} \nonumber\\
&\hspace{0 cm} \!-\!q^{17}\!+\! q^{18}\!-\!q^{19}\!+\!\frac{2 q\!+\!q^2\!+\!2 q^4\!+\! q^6\!+\! q^7\!+\! q^8}{1\!-\!q^9}
\!-\!\frac{1\!+\!q^2\!+\!q^5}{1\!-\!q^7}\!-\!\frac{q^{2}}{1\!-\!q^{11}} \Biggr)\succeq 0.
\end{align}
Note that
\begin{align*}
&\frac{2\!-\!q\!+\!2q^2\!-\!q^4\!+\!q^5\!-\!q^6\!-\!q^{10}\!-\!2 q^{11}\!+\!2 q^{12}\!+\!2 q^{14}\!-\!2 q^{15}\!+\!2 q^{16}\!-\!q^{17}\!+\!q^{18}\!-\!q^{19}}{1-q}\\
&=\frac{1-2q^{11}}{1-q}+\frac{q^2\left(1-q^8\right)}{1-q}+\frac{2q^{12}+q^{16}}{1-q}+1+q^2+q^3+q^5+2q^{14}+q^{16}+q^{18}.
\end{align*}
Since
\[
\frac{q^2\left(1-q^8\right)}{1-q}+\frac{2q^{12}+q^{16}}{1-q}+1+q^2+q^3+q^5+2q^{14}+q^{16}+q^{18}\succeq 0,
\]
to prove \eqref{eqn7b}, it suffices to show
\begin{align*}
&\frac{1}{1-q} \left(1-2q^{11} +\frac{2 q + q^2 + 2 q^4+q^6}{1-q^9}-\frac{1+ q^2+q^5}{1-q^7}-\frac{q^{2}}{1-q^{11}} \right)\succeq 0.
\end{align*}
To prove this, we prove the following
\begin{align}\label{eqn7d}
\frac{1}{1-q} \left(1-2q^{11} +\frac{q + q^2 + 2 q^4+q^6}{1-q^9}-\frac{1+ q^2+q^5}{1-q^7} \right)&\succeq 0,\\\label{eqn7e}
\frac{1}{1-q} \left(\frac{q}{1-q^9}-\frac{q^{2}}{1-q^{11}} \right)&\succeq 0.
\end{align}
Note that
\[
\frac{1}{1-q}\!\left(\frac{q}{1-q^9}\!-\!\frac{q^{2}}{1-q^{11}} \right)\!=\!\frac{q\!-\!q^{12}\!-\!q^2\!+\!q^{11}}{\left(1\!-\!q\right)\left(1\!-\!q^9\right)\left(1\!-\!q^{11}\right)}\!=\!\frac{q\!+\!q^{11}}{\left(1\!-\!q^9\right)\left(1\!-\!q^{11}\right)}\succeq 0,
\]
which proves \eqref{eqn7e}. 
Next we prove \eqref{eqn7d}. First we see that 
\begin{align*}
&\frac{1}{1-q} \left(1-2q^{11} +\frac{q + q^2 + 2 q^4+q^6}{1-q^9}-\frac{1+ q^2+q^5}{1-q^7} \right)\\
&\hspace{1 cm}\succeq 1+\frac{1}{1-q}\left(\frac{q + q^2 + 2 q^4+q^6-q^{11}+q^{20}}{1-q^9}-\frac{1+ q^2+q^5}{1-q^7}\right)\\
&\hspace{1 cm}=1+\frac{1}{1-q}\left(\frac{q + 2 q^4+q^6+q^{20}}{1-q^9}-\frac{1 + q^5 + q^{9}}{1-q^7}\right)\\
&\hspace{1 cm}\succeq 1+\frac{1}{1-q}\left(\frac{q + 2 q^4+q^6}{1-q^9}-\frac{1 + q^5 + q^{9}}{1-q^7}\right)\\
&\hspace{1 cm}=\frac{q^4\left(1+q+q^2+q^3+q^4\right)}{1-q^9}+\frac{1}{1-q}\left(\frac{1 + q^4+q^6+q^{10}}{1-q^9}-\frac{1 + q^5 + q^{9}}{1-q^7}\right).
\end{align*}
Since
\[
\frac{q^4\left(1+q+q^2+q^3+q^4\right)}{1-q^9}\succeq 0,
\]
to prove \eqref{eqn7d}, it is enough to show that
\begin{equation*}
\frac{1}{1-q}\left(\frac{1 + q^4+q^6+q^{10}}{1-q^9}-\frac{1 + q^5 + q^{9}}{1-q^7}\right)\succeq 0.
\end{equation*}


We write
\[
\frac{q^{10}}{1-q^9}=-q+\frac{q}{1-q^9},\quad \frac{q^{9}}{1-q^7}=-q^2+\frac{q^2}{1-q^7}.
\]
Thus we need to show that
\[
\frac{1}{1-q}\left(-q+q^2+\frac{1 + q+q^4+q^6}{1-q^9}-\frac{1 + q^2+q^5}{1-q^7}\right)\succeq 0.
\]
We combine 
\begin{align*}
\frac{1\!+\!q\!+\!q^4\!+\!q^6}{1\!-\!q^9}\!-\!\frac{1\!+\! q^2\!+\!q^5}{1\!-\!q^7}
&=\frac{\left(1\!-\!q\right) q \left(1\!+\!q^3\!+\!q^5\!-\!q^7\!-\!q^{12}\right)}{\left(1\!-\!q^7\right)\left(1\!-\!q^9\right)}.
\end{align*}
So we need
\[
-1+\frac{1 + q^3 + q^5 - q^7 - q^{12}}{\left(1-q^7\right)\left(1-q^9\right)}\succeq 0.
\]
It is not hard to see that the constant term of the left-hand side vanishes. Thus it is enough to show that
\[
\frac{1 + q^3  - q^7 - q^{12}}{\left(1-q^7\right)\left(1-q^9\right)}=\frac{1}{1-q^9}+\frac{q^3}{1-q^7}\succeq 0
\]
which holds. thus concludes the proof for $k=7$. \qed 

\phantom{.}

 Next, we prove \Cref{thm2}.

\noindent \emph{Proof of Theorem \ref{thm2}}.  Throughout the proof, $k\!\in\!\{8,9,10\}$. We begin by establishing an alternative form for $F_{k,1}(q)$.  To make the powers in the inner sum of  \Cref{newlem1} more explicitly positive, we distribute powers of $q$. For this, we write 
\begin{align*}
	\left(q^{4-2k};q^2\right)_n q^{(2k-1)n} 
= (-1)^n \left(q^{2k-2n-2};q^2\right)_n q^{n^2+2n}.
\end{align*}
Plugging this into \Cref{newlem1}, we obtain 
\begin{align*}
F_{k,1}(q) &= \frac{q}{(1-q)(q;q^2)_{k-1}} \left( 1 + (1-q) \sum_{n=1}^{k-2} \frac{(-1)^n \left(q^{2k-2n-2};q^2\right)_n q^{n^2+2n}}{(1-q^{2n+1})(q^2;q^2)_n} \right)\\
	&=\frac{1}{(q;q^2)_{k-1}}\!\sum_{n=0}^{k-2}\!\frac{(-1)^n \left(q^{2k-2n-2};q^2\right)_n q^{(n+1)^2}}{(1-q^{2n+1})(q^2;q^2)_n}\\
	&=\frac{1}{(q;q^2)_{k-1}}\!\sum_{n=0}^{k-2}\!\begin{bmatrix} k-2 \\ n \end{bmatrix}_{q^2}\!\!\! \frac{(-1)^n q^{(n+1)^2}}{1\!-\!q^{2n+1}}.
\end{align*}
Note that we can place all terms in the inner sum over the common denominator $1-q^\ell$, where $\ell := \text{lcm}(1,3,5,\dots,2k-3)$. Thus we obtain
\begin{equation*}
F_{k,1}(q) = \frac{1}{(q;q^2)_{k-1}(1-q^\ell)} \sum_{n=0}^{k-2} \begin{bmatrix} k-2 \\ n \end{bmatrix}_{q^2}  (-1)^n q^{(n+1)^2}\sum_{j=0}^{\frac{\ell}{2n+1}-1}q^{\left(2n+1\right)j}. 
\end{equation*}
Now $F_{k,1}(q)\succeq0$ holds if 
\begin{equation}\label{neweqn}
\frac{1}{1-q} \sum_{n=0}^{k-2} \begin{bmatrix} k-2 \\ n \end{bmatrix}_{q^2}  (-1)^n q^{(n+1)^2} \sum_{j=0}^{\frac{\ell}{2n+1}-1}q^{\left(2n+1\right)j} \succeq 0.
\end{equation}
Now consider the inner sum, which is the polynomial  
\begin{equation*}
	H_k(q) := \sum_{n=0}^{k-2} \begin{bmatrix} k-2 \\ n \end{bmatrix}_{q^2}  (-1)^n q^{(n+1)^2} \sum_{j=0}^{\frac{\ell}{2n+1}-1}q^{\left(2n+1\right)j}=:\sum_{j=0}^{d_k}a_k(j)q^j, 
\end{equation*}
where $d_k$ is the degree of $H_k(q)$, and $a_k(j):=\text{coeff}_{[q^j]}(H_k(q))$. 


For a polynomial $\mathcal{P}(q)=\sum_{j=0}^d a(j)q^j$ of degree $d$ ($d\in \mathbb{N}$) with $a(j)\in \mathbb{R}$ and $a(j)=0$ for $j>d$, we write 
\begin{align*}
\frac{\mathcal{P}(q)}{1-q}=\sum_{m\ge 0}\sum_{j=0}^d a(j)q^{j+m}=\sum_{n\ge 0} c(n) q^n,
\end{align*}
using $j+m=n$ and $a(n)>0$ for $n>d$ in the final step, where
\begin{equation*}
c(n):=\begin{cases}
\sum_{j=0}^{n} a(j)\ \ \text{if}\ \ 0\le n\le d,\\
\sum_{j=0}^{d} a(j)\ \ \text{if}\ \ n>d.
\end{cases}
\end{equation*}
Therefore, to prove $\frac{\mathcal{P}(q)}{1-q}\succeq 0$, it suffices to check that $c(n)\ge 0$ for $0\le n\le d$. Using this with $\mathcal{P}(q)= H_k(q)$, we see that to show \eqref{neweqn}, it is enough to check that $\sum_{j=0}^{m}a_k(j)\ge 0$ for all $0 \leq m \leq d_k$. We have checked this with a computer and the code\footnote{We provide the code for $k=10$. The code for $k\!\in\!\{8,9\}$ is similar.} is provided in Appendix A.\qed



Next, we prove Theorem \ref{lem2}. 

\noindent \emph{Proof of Theorem \ref{lem2}}.  By Lemma \ref{lem1}, it follows that
\begin{equation}\label{eqn5}
\lim_{k\to\infty} F_{k,1}(q)= q \sum_{n\ge 0}\frac{q^{n}}{\left(q;q^2\right)_{n+1}}.
\end{equation}
Applying Lemma \ref{RF} with $q\mapsto q^2$, $\alpha=0$, $\b=q^3$, and $w=q$, and then multiplying both sides by $\frac{q}{1-q}$ and finally using \eqref{Ramanujanomega}, we obtain
\begin{equation}\label{Omegareps}
q\sum_{n\ge 0}\frac{q^n}{\left(q;q^2\right)_{n+1}}
=q\omega(q).
\end{equation}
Plugging this into \eqref{eqn5}, we conclude the proof. \qed

We are now ready to prove Theorem \ref{lem3}.

\noindent \emph{Proof of Theorem \ref{lem3}}. By \eqref{Omegareps} and Lemma \ref{lem1}, we have
\begin{align}\label{eqn6}
q\omega(q)-F_{k,1}(q)&=\sum_{n\ge 0}\frac{q^{n+1}}{\left(q;q^2\right)_{n+1}}-\sum_{n\ge 0}\frac{\left(q^{2k-1};q^2\right)_nq^{n+1}}{\left(q;q^2\right)_{n+1}}\nonumber\\
&=\sum_{n\ge 0}\frac{1-\left(q^{2k-1};q^2\right)_n}{\left(q;q^2\right)_{n+1}}q^{n+1}.
\end{align} 
Applying \eqref{Gaussian} with $q\mapsto q^2, a=q^{2k-1}$, and $N=n$, it follows that
\[
\left(q^{2k-1};q^2\right)_n=\sum_{j=0}^{n}\begin{bmatrix}
n\\
j
\end{bmatrix}_{q^2}(-1)^j q^{j^2+2kj-2j}.
\]
Consequently, we have
\[
1-\left(q^{2k-1};q^2\right)_n=1-\sum_{j=0}^{n}\begin{bmatrix}
n\\
j
\end{bmatrix}_{q^2}(-1)^j q^{j^2+2kj-2j}=-\sum_{j=1}^{n}\begin{bmatrix}
n\\
j
\end{bmatrix}_{q^2}(-1)^j q^{j^2+2kj-2j}.
\]
Plugging this into \eqref{eqn6}, we get
\begin{align*}
&q\omega(q)-F_{k,1}(q)=\sum_{n\ge 0}\frac{q^{n+1}}{\left(q;q^2\right)_{n+1}}\sum_{j=1}^{n}\begin{bmatrix}
n\\
j
\end{bmatrix}_{q^2}(-1)^{j+1} q^{j^2+2kj-2j}\nonumber\\
&=\sum_{n\ge 0}\frac{q^{n}}{\left(q;q^2\right)_{n+1}}\sum_{j=1}^{n}\begin{bmatrix}
n\\
j
\end{bmatrix}_{q^2}(-1)^{j+1} q^{\left(j-1\right)^2+2kj}\nonumber\\
&
=\sum_{n\ge 1}\frac{q^{n}}{\left(q;q^2\right)_{n+1}}\sum_{j=0}^{n-1}\begin{bmatrix}
n\\
j+1
\end{bmatrix}_{q^2}(-1)^{j} q^{j^2+2kj+2k}\nonumber\\
&=q^{2k}\left(\frac{q}{\left(1-q\right)\left(1-q^3\right)}+\sum_{n\ge 2}\frac{q^{n}}{\left(q;q^2\right)_{n+1}}\sum_{j=0}^{n-1}\begin{bmatrix}
n\\
j+1
\end{bmatrix}_{q^2}(-1)^{j} q^{j^2+2kj}\right)\nonumber\\
&=q^{2k+1}\left(\frac{1}{\left(1-q\right)\left(1-q^3\right)}+\sum_{n\ge 2}\frac{q^{n-1}}{\left(q;q^2\right)_{n+1}}\sum_{j=0}^{n-1}\begin{bmatrix}
n\\
j+1
\end{bmatrix}_{q^2}(-1)^{j} q^{j^2+2kj}\right)\nonumber\\
&=q^{2k+1}\left(\frac{1}{(1-q)(1-q^3)}\right.\nonumber\\
&\hspace{3.4cm}\left.+\frac{1}{(1-q)(1-q^3)}\sum_{n\ge 1}\frac{q^{n}}{(q^5;q^2)_{n}}\sum_{j=0}^{n}\begin{bmatrix}
n+1\\
j+1
\end{bmatrix}_{q^2}(-1)^{j} q^{j^2+2kj}\right)\nonumber\\
&=\frac{q^{2k+1}}{\left(1-q\right)\left(1-q^3\right)}\left(1+\sum_{n\ge 1}\frac{q^{n}}{\left(q^5;q^2\right)_{n}}\sum_{j=0}^{n}\begin{bmatrix}
n+1\\
j+1
\end{bmatrix}_{q^2}(-1)^{j} q^{j^2+2kj}\right)\nonumber\\
&=q^{2k+1}E_k(q),
\end{align*}
where
\[
E_k(q):=\frac{1}{\left(1-q\right)\left(1-q^3\right)}\left(1+\sum_{n\ge 1}\frac{q^{n}}{\left(q^5;q^2\right)_{n}}\sum_{j=0}^{n}\begin{bmatrix}
n+1\\
j+1
\end{bmatrix}_{q^2}(-1)^{j} q^{j^2+2kj}\right).
\]

Next note that
$
E_k(q)\in \mathbb{Z}[[q]],
$
because 
\[
\frac{1}{\left(1-q\right)\left(1-q^3\right)}\in \mathbb{Z}[[q]]\ \ \text{and}\ \ \begin{bmatrix}
n+1\\
j+1
\end{bmatrix}_{q^2}\in\mathbb{Z}[q].
\]
Moreover, we have
\[
\operatorname{coeff}_{\left[q^0\right]}\left(\frac{1}{\left(1-q\right)\left(1-q^3\right)}\right)=1, 
\]
and
\[
\operatorname{coeff}_{\left[q^0\right]}\left(\sum_{n\ge 1}\frac{q^{n}}{\left(q^5;q^2\right)_{n}}\sum_{j=0}^{n}\begin{bmatrix}
n+1\\
j+1
\end{bmatrix}_{q^2}(-1)^{j} q^{j^2+2kj}\right)=0,
\]
because for $n,k\in\mathbb{N}$, $0\le j\le n$, $j^2+2kj+n\neq 0$. Consequently,
\[
\operatorname{coeff}_{\left[q^0\right]}E_k(q)=1.
\]
This concludes the proof.\qed





\section{Potential extensions of current approaches}\label{conc}

Each of the approaches employed in this paper runs into difficulties, which we now discuss, followed by some suggestions on extensions or alternatives.

In the proof of Theorem \ref{thm1}, for $k=5$, we first simplify $[\begin{smallmatrix}k-2\\n\end{smallmatrix}]_{q^2}$, then do long division and combine. However we are unable to find a uniform process of simplification, as shown in the proof of Theorem \ref{thm1} for $k\in\{6,7\}$. If a persistent pattern could be found then the entire conjecture might be provable in this way.

In the proof of Theorem \ref{thm2}, the calculation is algorithmic but the lcm grows exponentially in $k$. For $k=11$ we are unable to continue.  A more careful choice of where to cancel might allow further calculation; sufficient systematic insight, possibly including some combinatorial injections, might even be able to fully establish the conjecture. We make the following connection with $q$-binomial quotients, which may hold independent interest for some readers.  Recall \Cref{lem1}
\begin{equation}\label{gkn}
F_{k,1}(q)=\sum_{n\ge 0}\frac{\left(q^{2k-1};q^2\right)_nq^{n+1}}{\left(q;q^2\right)_{n+1}} =: \sum_{n \geq 0} G_{k,n} (q) q^{n+1}.
\end{equation}
We make the following refined conjecture, which would in turn imply \Cref{ABconj}. Next, we rewrite $G_{k,n}(q)$ as a quotient of $q$-binomials. 
\begin{equation}\label{qbinquot}
G_{k,n}(q) = \frac{1}{1-q^{2n+1}} \frac{\begin{bmatrix} 2k+2n-2 \\ 2n \end{bmatrix}_{q}}{\begin{bmatrix} k+n-1 \\ n \end{bmatrix}_{q^2}}.
\end{equation}

\begin{conjecture} We have $G_{k,n}(q) \succeq 0$.
\end{conjecture}

The form given in \eqref{qbinquot}, the symmetry of $q$-binomial coefficients about their central peaks, and computational data suggest the following further conjecture: 

\begin{conjecture}\label{gkndiffplusone} For $n\in\N$, $k \geq n+1$, we have $$G_{k,n} - G_{k+1,n-1} \succeq 0.$$  
\end{conjecture}

We can establish the following. 

\begin{lemma}\label{gknsmallk} For $n\in\N$, $1 \leq j \leq n$, $G_{n+2-j,n} = G_{n+2,n-j}$.
\end{lemma}

\begin{proof} 

Using \eqref{gkn} with $k= n-j+2$ and then using \eqref{gkn} with $k= n+2$ and $n\mapsto n-j$, we conclude the proof.	
\end{proof}

Since $G_{k,0}(q) = \frac{1}{1-q}$ has nonnegative coefficients for all $k$, it follows from Lemma \ref{gknsmallk} that Conjecture \ref{gkndiffplusone} would imply the nonnegativity of all $G_{k,n}$ and hence \Cref{ABconj}. Lemma \ref{gknsmallk} also has the consequence that if we can show that $G_{k,n} \succeq 0$ for each $k$ and $0 \leq n \leq k-2$, then all $G_{k,n} \succeq 0$ and Conjecture \ref{ABconj} would be proved. Finally, the forms exhibited in the proof of Theorem \ref{thm1}, along with computational evidence, further suggest the following conjecture, which would also be sufficient to establish the positivity of $F_{k,1}(q)$: 

\begin{conjecture}\label{fracconj} We have $G_{k,n}(q)=\sum_{j=0}^n \frac{z_{k,j}(q)}{1-q^{2j+1}}$ where all $z_{k,j}(q)$ have nonnegative coefficients.
\end{conjecture}

\appendix

\section{Code}

We include in this section the Maple code for the case $k=10$ that verifies Theorem \ref{thm2}. 

\vspace{-0.3 cm}

\phantom{.}

\verb!> eightpolys:=1, 1+q^2+q^4+q^6+q^8+q^10+q^12+q^14, 1+q^2+2q^4!

\verb! +2q^6+3q^8+3q^10+4q^12+3q^14+3q^16+2q^18+2q^20+q^22+q^24,!


\verb! 1+q^2+2q^4+3q^6+4q^8+5q^10+6q^12+6q^14+6q^16+6q^18+5q^20+4q^22 !

\verb! +3q^24+2q^26+q^28+q^30, 1+q^2+2q^4+3q^6+5q^8+5q^10+7q^12+7q^14 !
 
 \verb! +8q^16+7q^18+7q^20+5q^22+5q^24+3q^26+2q^28+q^30+q^32, 1+q^2 !
 
 
\verb! +2q^4+3q^6+4q^8+5q^10+6q^12+6q^14+6q^16+6q^18+5q^20+4q^22+3q^24!

\verb! +2q^26+q^28+q^30, 1+q^2+2q^4+2q^6+3q^8+3q^10+4q^12+3q^14+3q^16!
 

\verb! +2q^18+2q^20+q^22+q^24, 1+q^2+q^4+q^6+q^8+q^10+q^12+q^14, 1!


\verb!> tenbigpoly := (eightpolys[1]) (q^((0+1)^2))((-1)^0)!

\quad \quad \verb!add(q^(j(2(0) + 1)), j=0..((765765/(2(0)+1))-1))!

\verb! + (eightpolys[1+1]) (q^((1+1)^2))((-1)^1)!

\quad \quad \verb!add(q^(j(2(1) + 1)), j=0..((765765/(2(1)+1))-1))!

\verb! + (eightpolys[2+1]) (q^((2+1)^2))((-1)^2)!

\quad \quad \verb!add(q^(j(2(2) + 1)), j=0..((765765/(2(2)+1))-1))!

\verb! + (eightpolys[3+1]) (q^((3+1)^2))((-1)^3)!

\quad \quad \verb!add(q^(j(2(3) + 1)), j=0..((765765/(2(3)+1))-1))!

\verb! + (eightpolys[4+1]) (q^((4+1)^2))((-1)^4)!

\quad \quad \verb!add(q^(j(2(4) + 1)), j=0..((765765/(2(4)+1))-1))!

\verb! + (eightpolys[5+1]) (q^((5+1)^2))((-1)^5)!

\quad \quad \verb!add(q^(j(2(5) + 1)), j=0..((765765/(2(5)+1))-1))!

\verb! + (eightpolys[6+1]) (q^((6+1)^2))((-1)^6)!

\quad \quad \verb!add(q^(j(2(6) + 1)), j=0..((765765/(2(6)+1))-1))!

\verb! + (eightpolys[7+1]) (q^((7+1)^2))((-1)^7)!

\quad \quad \verb!add(q^(j(2(7) + 1)), j=0..((765765/(2(7)+1))-1))!

\verb! + (eightpolys[8+1]) (q^((8+1)^2))((-1)^8)!

\quad \quad \verb!add(q^(j(2(8) + 1)), j=0..((765765/(2(8)+1))-1)):!

\verb!> currsum :=0;!

\quad \verb!errorfound:=False;!

\quad \verb!for s from 1 to 765830 do!

\quad \verb!currsum := currsum + coeff(tenbigpoly, q^s):!

\quad \verb!if currsum < 0 then!

\quad \verb!print("Error found at", s);!

\quad \verb!errorfound := True;!

\quad \verb!end if;!

\quad \verb!end do:!

\quad \verb!if errorfound = False then!

\quad \verb!print("No error found")!

\quad \verb!end if!

\phantom{.}

\end{document}